\documentclass[11pt]{article}

\usepackage[]{algorithm2e}
\usepackage{amsthm}
\usepackage{amssymb}
\usepackage{amsmath}
\usepackage[top = 3cm, left=2.5cm, right=2.5cm]{geometry}
\usepackage[english]{babel}
\usepackage[nodayofweek,level]{datetime}
\usepackage{fancyhdr}
\usepackage{dirtytalk}
\usepackage{parskip}
\usepackage{graphicx}
\usepackage{xcolor}

\usepackage{enumitem}   

\newtheorem{theorem}{Theorem}[section]

\newtheorem{lemma}[theorem]{Lemma}
\newtheorem{definition}[theorem]{Definition}
\newtheorem{claim}[theorem]{Claim}
\newtheorem{remark}[theorem]{Remark}
\newtheorem{conjecture}[theorem]{Conjecture}

\newcommand{\EE}{\ensuremath{\mathbf{E}}}
\newcommand{\Var}{\ensuremath{\text{Var}}}

\newcommand\numberthis{\addtocounter{equation}{1}\tag{\theequation}}

\usepackage{hyperref}
\hypersetup{
	colorlinks,
	linkcolor={red!60!black},
	citecolor={green!50!black},
	urlcolor={blue!80!black}
}

\title{Maker-Breaker Games on Random Hypergraphs}

\author{Maxime Larcher\thanks{%
	Department of Computer Science, ETH Z\"urich, Switzerland%
	\newline%
	larcherm@inf.ethz.ch} 
}

\begin{document}

\maketitle

\abstract{In this paper, we study Maker-Breaker games on the random hypergraph \(H_{n,s,p}\), obtained from the complete \(s\)-graph by keeping every edge independently with probability \(p\). We determine the threshold probability for the property of Maker winning the game as a function of \(s\), the uniformity of the underlying hypergraph, as well as \(m\), \(b\), the number of vertices that Maker and Breaker are respectively allowed to pick each turn.

In addition, we show that depending on those \(m,b,s\), there are two types of thresholds: either being Maker-win is a local property and the threshold is weak, or it is related to global properties of the random hypergraph and the threshold is semi-sharp. We conjecture that in the latter case, the threshold is actually sharp.}

\section{Introduction}

A Maker-Breaker game is a two-player, perfect information game played on a hypergraph \(H = (V,E)\) called the board of the game. The two players, Maker and Breaker, alternately pick vertices of \(V\), until none are left unpicked. Maker wins the game if at the end, the subset of vertices he has picked contains an edge \(e \in E\); otherwise Breaker wins. We let \(m, b\) denote the number of vertices Maker and Breaker are respectively allowed to pick at each turn. A Maker-Breaker game with parameters \((m, b)\) is called \emph{a \((m,b)\)-game}. We generally assume that Maker starts picking vertices.

If Maker (resp.\ Breaker) has a winning strategy, we say that the game is a \emph{Maker-win} (resp.\ a \emph{Breaker-win}), or simply that \emph{Maker wins} (resp.\ \emph{Breaker wins}) on this game. Given some game, a natural question is to ask who of Maker or Breaker wins on it. One of the earliest partial results to this question is the \emph{Erd\"os-Selfridge criterion}~\cite{erdos1973combinatorial}, which was generalised by Beck~\cite{beck1982remarks}.
 
\begin{theorem}[Erd\"os, Selfridge~\cite{erdos1973combinatorial}, Beck~\cite{beck1982remarks}]
	Let \(H = (V,E)\) be a hypergraph and consider the \((m,b)\)-game on this hypergraph in which Maker plays first. Assume that 
		\begin{align*}
			\sum_{e \in E}{(1 + b)^{-|e|/m}} < (1 + b)^{-1}, \label{eq:erdos_selfridge_criterion} \numberthis
		\end{align*} 
	then Breaker has a winning strategy.
\end{theorem}

Continuing work of Chv\'atal and Erd\"os~\cite{chvatal1978biased}, Beck~\cite{beck1982remarks, beck1985random} studied specific games on the complete graph \(K_n\) where Maker's goal is to build a clique, a Hamiltonian cycle, a spanning tree, etc. Note that in this case, the vertices of the board are the edges of \(K_n\), and the edges of the board are the structures Maker wants to build (cliques, Hamiltonian cycles, etc.). Using the criterion above, together with a similar criterion for Maker, he determined the winner for those \((1,b)\)-games for a wide range of \(b\).

Since the seminal work of Erd\"os and R\'enyi who introduced the concept of random graphs (see e.g.\ \cite{bollobas2001random} for an introduction to this theory), there has been an ever-growing interest for the the properties of the random graph \(G_{n,p}\) and their respective thresholds. For a graph property \(\mathcal{P}_n\) indexed by \(n\), a weak threshold is a function \(p_0(n)\) such that \(\{G_{n,p} \in \mathcal{P}_n\}\) holds w.h.p.\footnote{With high probability, or w.h.p.\ for short, means with probability tending to \(1\) as \(n\) tends to infinity.}\ if \(p = \omega(p_0)\), and \(\{G_{n,p} \notin \mathcal{P}_n\}\) holds w.h.p.\ if \(p = o(p_0)\). The thresholds is semi-sharp if there exist constants \(0 < C_1 \le C_2\) such that the same holds respectively for \(p \le (C_1 - \varepsilon) p_0\) or \(p \ge (C_2 - \varepsilon) p_0\) for all postive \(\varepsilon\). It is sharp if \(C_1 = C_2\).

As part of this work, Stojakovi{\'c} and Szab{\'o}~\cite{stojakovic2005positional} studied those same games as Beck, but now played \(G_{n,p}\). Adapting techniques of Chv\'atal and Erd\"os~\cite{chvatal1978biased} as well as the criterion above, they computed the threshold \(p_0(n, b)\) for Maker winning the game as a function of \(b\). Since then, various papers have studied related questions and other games in which randomness is involved. For example, and to only cite a few, Clemens, Ferber, Krivelevich and Liebenau~\cite{clemens2012fast} study how fast Maker is able to win those games on \(G_{n,p}\), while Krivelevich and Kronenberg~\cite{krivelevich2015random} study games in which it is not the board that is random but rather who plays at each turn.  

In this paper, we also study random Maker-Breaker games but from a new point of view. Instead of looking at the hypergraph obtained from some subgraphs of \(G_{n,p}\), we take a look at the `Erd\"os-R\'enyi random hypergraph' \(H_{n,s,p}\), that is the hypergraph obtained from the complete \(s\)-uniform hypergraph by keeping each edge with probability \(p\), \emph{independently of other edges}. Our main contribution is to determine threshold for Maker winning the game on this hypergraph.

\begin{theorem}
	\label{thm:threshold}	
	Let \(m,b,s\) be fixed constants and let \(p_0\) denote the weak threshold for the property of Maker winning the \((m,b)\)-game on \(H_{n,s,p}\). Then: \begin{enumerate}[label=(\roman*)]
		\item if \(s \le m\) or \(b < m\), then \(p_0 = n^{-s}\); \label{itm:thm_threshold_1}
		\item if \(m = s-1 \le b\), then \(p_0 = n^{1 - s - \left\lceil \frac{b+1}{m} \right\rceil^{-1}}\); \label{itm:thm_threshold_2} 
		\item if \(m \le s-2, m \le b\), then \(p_0 = n^{1-s}\). Additionally, the threshold is semi-sharp. \label{itm:thm_threshold_3} 
	\end{enumerate}

\end{theorem}

As mentioned earlier, we assume that Maker makes the first move. It is nonetheless straightforward to see that most statements, and in particular Theorem~\ref{thm:threshold} above, hold if we assume that Breaker starts. 

Before going into more specifics, we note that in our case, as \(m, b, s\) are constants, Erd\"os-Selfridge-Beck Criterion only guarantees that Breaker wins if there is a finite number of edges. This is of little help, and for this reason, the arguments developed here differ from those presented in the aforementioned papers.

This paper is divided as follows. In Section~\ref{sec:preliminary} we introduce some definitions and make some elementary remarks. In Section~\ref{sec:easy_case} we prove points \ref{itm:thm_threshold_1} and \ref{itm:thm_threshold_2} of Theorem~\ref{thm:threshold}. The proofs of those two points are fairly similar in the sense that they both rely on local properties of \(H_{n,s,p}\). In Section~\ref{sec:hard_case} we prove point~\ref{itm:thm_threshold_3} of Theorem~\ref{thm:threshold}. Unlike \ref{itm:thm_threshold_1} and \ref{itm:thm_threshold_2}, this proof relies on global properties of the graph.

\section{Preliminary}
\label{sec:preliminary}

We denote the natural logarithm by \(\log\). Whenever they are not essential, we omit the floor \(\lfloor \cdot \rfloor\) and ceil \(\lceil \cdot \rceil\) signs from our computations. 

We use standard graph-theoretic notations (see e.g.\ \cite{diestel2017graph}). A hypergraph is a pair \(H = (V, E)\) where \(V\) is an arbitrary finite set called \emph{vertex set} and \(E \subseteq 2^V\) is called \emph{edge set}. Elements of \(V\) and \(E\) are respectively called vertices and edges. When dealing with several hypergraphs and whenever there may be some confusion, we write \(V(H), E(H)\) to emphasise that we are talking about the vertices or edges of hypergraph \(H\). We say that \(H\) is \(s\)-uniform or a \(s\)-graph if all its edges have size \(s\). The degree of a vertex \(v\), which we write \(d(v)\), is the number of edges containing \(v\).

In Section~\ref{sec:hard_case} it will be convenient to not look at the game on \(H_{n,s,p}\) but at a related game, which is `easier' for Maker. Intuitively, if a game has more and shorter edges than another game, then it should be easier for Maker to win on the former. The definition below formalises this idea.
\begin{definition}
	Let \(H = (V,E)\), \(H' = (V,E')\) be two hypergraphs (not necessarily uniform) on the same vertex set. We say that \(H\) is easier than \(H'\) and write \(H \le H'\) if for all \(e' \in E'\) there exists \(e \in E\) such that \(e \subseteq e'\).  
\end{definition}

So that it is straightforward to use in Section~\ref{sec:hard_case}, we make the following remark, which is a just a consequence of the definition above.
\begin{remark}
	\label{rk:harder_means_harder} If \(H \le H'\) and Breaker has a winning strategy for the \((m,b)\)-game on \(H\), then the same strategy is also a winning strategy for the \((m,b)\)-game on \(H'\).
\end{remark}

Our proof of Theorem~\ref{thm:threshold}~\ref{itm:thm_threshold_3} also uses results of Karo{\'n}ski and {\L}uczak~\cite{karonski2002phase} so we give the following definitions from their paper.
\begin{definition}[Connectedness]
	Let \(H = (V,E)\) be a hypergraph. We say that \(H\) is connected if for all \(x, y \in V\) there exists edges \(e_1, \dots e_k\) such that \(x \in e_1\), \(y \in e_k\) and \(e_i \cap e_{i+1} \ne \emptyset\) for all \(i \in \{1, \dots, k-1\}\).
\end{definition}

\begin{definition}[Excess]
	Let \(H = (V, E)\) be a \(s\)-graph. The excess of \(H\) is defined as \[ex(H) = (s-1)|E| - |V|.\]
	If \(H\) is connected and \(ex(H) = -1\), we say that \(H\) is a \emph{tree}; if \(H\) is connected and \(ex(H) = 0\), we say that \(H\) is a \emph{unicycle}.
\end{definition}

Trees and unicyclic hypergraphs essentially behave like (standard) trees and unicyclic graphs so we give the following results without proof.

\begin{claim}
	\label{clm:excess_of_union}
	If \(H\), \(H'\) are two \(s\)-graphs on disjoint vertex sets, then the excess of their union is the sum of their excess. In particular, if a \(s\)-graph \(H\) is a disjoint collection of trees and unicycles, then its excess is at most \(0\).
		
	If \(H\) is a collection of disjoint trees and unicycles, then so is any \(H' \subseteq H\).
\end{claim}

\section{Proof of Theorem~\ref{thm:threshold}~\ref{itm:thm_threshold_1} and \ref{itm:thm_threshold_2}}

\label{sec:easy_case}

To prove points \ref{itm:thm_threshold_1}, \ref{itm:thm_threshold_2} we will study local properties of the random hypergraph \(H_{n,s,p}\). More precisely, we will show that when \(p\) is under the threshold, then no vertex has high degree and, at each turn, Breaker is able to `kill' all edges in which Maker has taken a vertex. On the other hand, when \(p\) is above the threshold, some subgraphs of constant size on which Maker wins appear in \(H_{n,s,p}\). 

First, we introduce the following graph which will be quite helpful.
\begin{definition}
	The \(d\)-star is the \(s\)-graph composed of \(d\) edges, all intersecting in a unique, common vertex called the \emph{centre}.
\end{definition}

Note in particular that the \(1\)-star is just a single edge. We can now summarise the ideas above in the following lemma.

\begin{lemma}
	\label{lem:degree_threshold}
	Let \(d\) be a positive integers. \begin{enumerate}[label=(\roman*)]
		\item If \(p = o \left( n^{1 - s - d^{-1}} \right)\), then w.h.p.\ \(H_{n,s,p}\) contains no vertex of degree at least \(d\); \label{itm:degree_threshold_1}
		\item if \(p = \omega \left( n^{1 - s - d^{-1}} \right)\), then w.h.p.\ \(H_{n,s,p}\) contains \(\omega(1)\) disjoint \(d\)-stars.\label{itm:degree_threshold_2}
	\end{enumerate}
\end{lemma}

Before we prove this lemma, we show how it implies \ref{itm:thm_threshold_1} and \ref{itm:thm_threshold_2} of Theorem~\ref{thm:threshold}.

\begin{proof}[Proof of Theorem~\ref{thm:threshold}~\ref{itm:thm_threshold_1} and \ref{itm:thm_threshold_2}]
	
	We start by proving \ref{itm:thm_threshold_1}. If \(s \le m\), Maker wins as long as there is an edge; if \(b < m < s\), the work of Hamidoune and Las Vergnas~\cite{hamidoune1987solution} implies that there exists \(M = M(m,b,s)\) such that Maker wins on \(M\) disjoint \(s\)-uniform edges. Applying Lemma~\ref{lem:degree_threshold} with \(d=1\), we see that when \(p = o(n^{-s})\) w.h.p.\ there are no edges in \(H_{n,s,p}\) and when \(p = \omega(n^{-s})\) w.h.p.\ there are \(\omega(1) \ge M\) disjoint edges. So \(p_0 = n^{-s}\) is indeed a threshold when \(s \le m\) or \(b < m < s\).
	
	For \ref{itm:thm_threshold_2}, observe that if there are only vertices of degree (strictly) smaller than \( \left\lceil \frac{b+1}{m} \right\rceil \), then after each turn, for any choice of \(m\) vertices by Maker, there are at most \(m \left( \left\lceil \frac{b+1}{m} \right\rceil - 1 \right) \le b\) edges containing a vertex of Maker and Breaker can delete them all. Now suppose there exist \(m\) disjoint \(d\)-stars with \(d = \left\lceil \frac{b+1}{m} \right\rceil\). Then Maker can take the centres of each on their first turn. This leaves at least \(b+1\) disjoint edges and Breaker cannot pick a vertex in each. Maker simply take the \(s-1 = m\) remaining vertices in one of those and wins. As shown by Lemma~\ref{lem:degree_threshold}, if \(p = o\left( n^{1 - s - \left\lceil \frac{b+1}{m} \right\rceil} \right)\) then w.h.p.\ we are in the first case and Breaker wins, and if \(p = \omega\left( n^{1 - s - \left\lceil \frac{b+1}{m} \right\rceil} \right)\) w.h.p.\ we are in the second and Maker wins.
\end{proof}

The rest of this section is dedicated to proving Lemma~\ref{lem:degree_threshold}.

\begin{proof}[Proof of Lemma~\ref{lem:degree_threshold}]
	We start by proving \ref{itm:degree_threshold_1}. For each vertex \(v\), we denote by \(X_v\) the indicator random variable that \(v\) has degree at least \(d\), so that \(X = \sum_{v\in V}{X_v}\) counts the number of vertices of degree at least \(d\). 
	Consider a fixed \(v\). There are \(n \choose s-1\) potential edges containing it, each present independently of others with probability \(p\), so that 
		\[\EE[X_v] = \sum_{k = d}^{n-1 \choose s-1}{ {{n-1 \choose s-1} \choose k } p^k (1 - p)^{ {n-1 \choose s-1} - k }  } = O \left( n^{d(s-1)} p^d \right).\] 
	By linearity, \(\EE[X] \le O \left( n^{1 + d(s-1)} p^d \right) = o(1)\) and by Markov inequality, we conclude that \(\Pr[ X \ge 1 ] \le \EE[X] \le o(1) \).

	We now turn our attention to \ref{itm:degree_threshold_2}. First observe that the maximum number of non-intersecting \(d\)-stars is monotone in \(p\), so it suffices to prove the claim for all \(\omega \left( n^{1 - s - d^{-1}} \right) \le p \le o \left( n^{1 - s - d^{-1} + \varepsilon} \right)\) where \(\varepsilon\) is some positive constant. Here we choose \(\varepsilon = 1/2d^2\).
	
	Let us introduce the following random variables: \(Y\) denotes the total number of copies of \(d\)-stars and \(Z\) denotes the number of pairs of distinct, but intersecting \(d\)-stars. By considering the set of all stars and arbitrarily removing one star for each intersecting pair, one obtains a set of disjoint stars. Hence to prove the claim, it is sufficient to show that \(Y = \omega(1)\) w.h.p.\ and \(Z = 0\) w.h.p.
	
	First, we show that w.h.p.\ \(Z = 0\). Observe that two distinct but intersecting stars may share \(j = 0, \ldots, d-1\) edges and need share \(i \ge 1 + (s-1)j\) vertices. We may decompose \(Z\) into 
		\[Z = \sum_{(i,j)}{Z_{i,j}},\] 
	where \(Z_{i,j}\) counts the number of pairs of stars sharing \(i\) vertices and \(j\) edges. 
	For a fixed \((i, j)\) and for a fixed subset of \(2( 1 + (s-1)d) - i\) vertices, there is a bounded number of ways to embed two stars sharing \(i\) vertices, \(j\) edges. Each such embedding has probability \(p^{2d - j}\) of being present in \(H_{n,s,p}\). Hence, we have 
		\begin{align*}
			\EE[Z_{i,j}]
				&= O \left( n^{2(1 - (s-1)d) - i} p^{2d - j} \right) \\
				&= O \left( n^{2(1 - (s-1)d)} p^{2d} \right) \cdot \Theta \left( n^{-i} p^{-j} \right). \numberthis \label{eq:degree_threshold_expY}
		\end{align*}
	
	Note that because we chose \(p \le o \left( n^{1 - s - d^{-1} + \varepsilon} \right)\) the first term on the RHS of \eqref{eq:degree_threshold_expY} is \(\Theta \left( n^{2(1 - (s-1)d)} p^{2d} \right) = o(n^{2 d \varepsilon}).\)
	Also, since \(i \ge 1 + (s-1)j\) and \(p \ge \omega \left( n^{1 - s - d^{-1}} \right)\), the second term on the RHS of \eqref{eq:degree_threshold_expY} is \(\Theta\left( n^{-i}p^{-j} \right) = o( n^{-1 + j/d} ) = o(n^{-1/d}).\) Combining those two and recalling that we chose \(\varepsilon = 1/2d^2\), we find \(\EE[Z_{i,j}] = o(1)\). The set of possible \((i,j)\) is of constant size, so we deduce that 
		\begin{align*}
			\EE[Z] = \sum_{(i,j) \in I}{\EE[Z_{i,j}]} = o(1),\label{eq:degree_threshold_expY_approx} \numberthis
		\end{align*}
		and conclude using the first moment method.
	
	Let us now prove that w.h.p.\ \(Y = \omega(1)\). For each \(d\)-star \(S\) in the complete \(s\)-graph, we denote by \(Y_S\) the indicator random variable that \(S\) is in \(H_{n,s,p}\). There are \(\Theta( n^{1 + (s-1)d} )\) such \(S\), each has probability \(p^d\) of being in \(H_{n,s,p}\). By linearity of expectation, it follows that the expected total number of copies of stars in \(H_{n,s,p}\) is 
		\begin{align*}
			\EE[Y] = \sum_{S}{\EE[Y_S]} = \Theta( n^{1 + (s-1)d} p^d ) = \omega(1), \numberthis \label{eq:degree_threshold_expX}
		\end{align*} since \(p = \omega \left( n^{1 - s - d^{-1}} \right)\).
	We wish to apply the second moment method to conclude. For that, we observe that
		\begin{align*}
			\Var Y \le \EE[Y] + \EE[Z]. \label{eq:degree_threshold_varX_rel} \numberthis
		\end{align*}
	To see this, recall that \((Y_S)_S\) are indicator random variables, so
		\[\Var Y = \sum_{S, S'}{ \EE[Y_S Y_{S'}] - \EE[Y_S] \EE[Y_{S'}] } \le \sum_{S \cap S' \neq \emptyset}{ \EE[Y_S Y_{S'}] } = \EE[Y] + \EE[Z]. \]
	
	Combining \eqref{eq:degree_threshold_expY_approx}, \eqref{eq:degree_threshold_expX} and \eqref{eq:degree_threshold_varX_rel}, we have \(\Var Y = o(\EE[Y]^2)\) and conclude that \(Y = \omega(1)\) w.h.p.\ using the second moment method.
	
\end{proof}

\section{Proof of Theorem~\ref{thm:threshold}~\ref{itm:thm_threshold_3}}
\label{sec:hard_case}

The approach in this section will differ substantially from the one in the previous section. In the case \(m \le s-2\), \(m \le b\), it appears that being Maker- or Breaker-win is not longer a local property but rather a global property. Let us detail: in the previous section, we showed that, roughly speaking, Maker started winning on \(H_{n,s,p}\) as soon as some game --- a collection of finitely many edges or \(d\)-stars --- started appearing. The same phenomenon does not happen in the case \(m \le s-2, m \le b\). Actually, and as we discuss further in Section~\ref{sec:conclusion}, when \(p = \Theta(n^{1-s})\), any given Maker-win game has low probability of being present. 

So instead of looking for local properties, we identify two global properties of \(H_{n,s,p}\) whose (sharp) thresholds are of order \(n^{1-s}\) and show that they guarantee, respectively, that Maker or Breaker wins.

The first of those global properties is the equivalent for hypergraphs of the well-known fact (see e.g.~\cite{bollobas2001random}) that when the number of edges is low, the typical random graph only contains trees and unicyclic components.

\begin{theorem}[Karo\'nski, \L{}uczak~\cite{karonski2002phase}]
	\label{thm:tree_unicycle_decomp}
	There exists a constant \(0 < c_1\) such that if \(p \le c_1 n^{1-s}\) then w.h.p.\ \(H_{n,s,p}\) is a disjoint collection of trees and unicycles.
\end{theorem}

We show that on such a game, Breaker has a winning strategy. 

\begin{lemma}
	\label{lem:Bwin_tree_unic}
	If \(m \le s-2, m \le b\) then Breaker wins on any collection of disjoint trees and unicycles.
\end{lemma}

\begin{proof}
	The idea is to show that if \(H = (V,E)\) is a collection disjoint trees and unicycles, there is a collection of disjoint \((s-1)\)-uniform edges \(H'\) on \(V\) such that \(H' \le H\). Since \(H'\) contains disjoint edges of size \(s-1 > m\) and \(b \ge m\), Breaker wins on \(H'\). Remark~\ref{rk:harder_means_harder} implies that Breaker also wins on \(H\).
	
	To show that such a \(H'\) exists, we take a look at a flow problem on a related auxiliary graph. Consider the directed graph \(G\) whose vertices are a source, a sink as well as the edges and vertices of \(H\), i.e.\ \(V(G) = \{v_{source}, v_{sink}\} \cup E(H) \cup V(H)\). In \(G\), there are edges \((v_{source},e)\) of capacity \((s-1)\) from the source to all \(e \in E(H)\), edges \((v, v_{sink})\) of capacity \(1\) from all \(v \in V(H)\), and finally, all edges \((e, v) \in E(H) \times V(H)\) with capacity \(1\) whenever \(v \in e\). 
	
	We claim that the value of the maximum flow of this network is \((s-1)|E(H)|\). To see this, let \( C = (S, V(G) \setminus S) \) be any cut in \(G\) such that \(v_{source} \in S, v_{sink} \notin S\). There exist \(\tilde{E} \subseteq E(H), \tilde{V} \subseteq V(H)\) such that \(S = \{v_{source}\} \cup \tilde{E} \cup \tilde{V}\) and we let \(N(\tilde{E})\) be the outneighbourhood of \(\tilde{E}\). By definition of \(G\), \(N( \tilde{E}) = \bigcup_{e \in \tilde{E}}{e}\) corresponds to the support of the edges of \(\tilde{E}\) in \(H\). The hypergraph \(\tilde{H} = ( N( \tilde{E} ), \tilde{E} )\) is a subgraph of \(H\), so by Claim~\ref{clm:excess_of_union} it is a collection of disjoint trees and unicycles and its excess is at most \(ex(\tilde{H}) \le 0\), which implies 
		\begin{align*}
			|N( \tilde{E} )| \ge (s-1)|\tilde{E}|. \label{eq:flow_hall_cond} \numberthis
		\end{align*} 
	
	The value of the cut \(C\) is the sum of capacities of edges \((v_{source}, e)\) when \(e \notin \tilde{E}\), edges \((v, v_{sink})\) when \(v \in \tilde{V}\) and edges \((e, v)\) when \(e \in \tilde{E}, v \notin \tilde{V}\). Using \eqref{eq:flow_hall_cond}, we find that the value of \(C\) is at least
		\begin{align*}
			(s-1)\left(|E(H)| - |\tilde{E}| \right) + |\tilde{V}| + \left( |N(\tilde{E})| - |\tilde{V}| \right) \ge (s-1)|E(H)|.
		\end{align*} 
	By max-flow min-cut Theorem, we conclude that the value of the maximum flow \(f\) is indeed \((s-1) |E(H)|\). For each \(e \in E(H)\), there exist unique \(v^e_1, \dots, v^e_{s-1} \in e\) such that \(f(e,v^e_1) = \cdots = f(e, v^e_{s-1}) = 1\). We let \(H'\) be the \((s-1)\)-graph on \(V(H)\) where the edge set is the collection of all those \(\left(\{v^e_1, \dots, v^e_{s-1}\}\right)_{e \in E(H)}\). It is straightforward to check that, as claimed above, the edges of \(H'\) are disjoint and \(H' \le H\).

\end{proof}

With this Lemma, we are now ready to prove Theorem~\ref{thm:threshold}~\ref{itm:thm_threshold_3}.

\begin{proof}[Proof of Theorem~\ref{thm:threshold}~\ref{itm:thm_threshold_3}]
	
By Theorem~\ref{thm:tree_unicycle_decomp} and Lemma~\ref{lem:Bwin_tree_unic}, there exists a constant \(c_1 > 0\) such that when \(p \le c_1 n^{1-s}\) Breaker wins w.h.p.

Suppose now that \(p \ge c_2 n^{1-s}\) for some constant \(c_2\) large enough. Because at each turn Maker picks \(m\) vertices and Breaker picks \(b\), at the end of the game Maker has taken at least \(t = \frac{mn}{m+b}\) vertices. We prove that all subsets of size \(t\) contain at least an edge, so that, no matter which strategy Maker follows, he wins w.h.p. 

Let \(T \subseteq V\) be a subset of \(|T| = t\) vertices. The probability that \(T\) contains no edge is \[\Pr[T \text{ contains no edge}] = (1-p)^{t \choose s} \le e^{- \frac{p t^s}{s!} (1 - o(1))}.\]
As \(t = \frac{mn}{m+b}\), we have \(t^s = t \cdot \left( \frac{m}{m+b} n \right)^{s-1} \). Also \(p = c_2 n^{1-s}\), so the expression above can be rewritten as 
	\[\Pr[T \text{ contains no edge}] \le e^{ - \left( \frac{c_2}{s!} (\frac{m}{m+b})^{s-1} - o(1) \right) t }.\]
There are \({n \choose t} \le \left( \frac{en}{t} \right)^t = \left( e\frac{m+b}{m} \right)^t\) subsets of \(t\) vertices, so the expected number of those which do not contain an edge is at most 
	\begin{align*}
		\left(e \frac{m+b}{m} \cdot e^{ - \left( \frac{c_2}{s!} (\frac{m}{m+b})^{s-1} - o(1) \right)} \right)^t,
	\end{align*} which, for a choice of \(c_2\) large enough, decays to \(0\) as \(n\) goes to infinity. We conclude using the first moment method.
\end{proof}

\section{Conclusion and Remarks}

\label{sec:conclusion}

In this paper, we have determined the threshold probability for which the \((m,b)\)-game on the random hypergraph \(H_{n,s,p}\) is won by Maker. When \(b < m\) or \(m \le s-1\), this threshold corresponds to the threshold at which some `small' game appears. In contrast, when \(m \le s-2, m \le b\), Maker winning on \(H_{n,s,p}\) is tied to global properties.

One may ask whether the threshold in that last case is a threshold at which some specific Maker-win \(W_n\) (possibly depending on \(n\)) starts appearing. The answer to this question is no. If Maker wins on \(W_n\), then some subgraph \(W'\) of it needs to have excess at least \(ex(W') \ge 1\), as otherwise Claim~\ref{clm:excess_of_union} guarantees that \(W_n\) is a collection of trees and unicycles, which by Lemma~\ref{lem:Bwin_tree_unic} is not a Maker-win. However, when \(p = \Theta(n^{1-s})\), such a \(W'\) has probability \(O\left( n^{|V(W')|} \cdot p^{|E(W')|} \right) = O \left( n^{- ex(W')} \right) = O(n^{-1})\) of being in \(H_{n,s,p}\).

One may also ask whether the weak thresholds of Theorem~\ref{thm:threshold}~\ref{itm:thm_threshold_1} and \ref{itm:thm_threshold_2} could be improved to sharp thresholds. Again, this is not the case: it is quite easy to see that when \(p = \Theta( p_0 )\), the probability that Maker wins is constant, bounded away from \(0\) and \(1\).

Finally, one may ask if the semi-sharp threshold of Theorem~\ref{thm:threshold}~\ref{itm:thm_threshold_3} can be improved to a sharp threshold. We conjecture that this is true.

\begin{conjecture}
	When \(m \le s-2, m \le b\), the threshold is sharp: there exists \(C > 0\) such that, for all \(\varepsilon > 0\), if \(p \le (C - \varepsilon)n^{1-s}\), Breaker wins w.h.p.\ and if \(p \ge (C + \varepsilon) n^{1-s}\), Maker wins w.h.p.
\end{conjecture}

Friedgut~\cite{friedgut2005hunting} gives an overview of how one may approach such a problem. He illustrates by giving a proof that the threshold for hypergraph \(2\)-colourability is sharp. Although the arguments we use to prove the semi-sharpness of Maker-win are somewhat similar to those of Friedgut, the key argument which allows him to conclude to the sharpness does not translate to our setting. We believe that proving the sharpness of Maker-win property requires new ideas.

As a short and final concluding remark, we note that unlike most previous papers on Maker-Breaker games, we focused on rather sparse games with small edges. It would interesting to know how the threshold probability behaves if we allow \(s\) to grow with \(n\).

\bibliographystyle{abbrv}
\bibliography{random_mb_games}

\end{document}